\documentclass[12pt]{amsart}
\usepackage{latexsym}
\usepackage{rotating}
\usepackage{amsfonts}
\usepackage{amssymb,amscd}
\usepackage{amsmath}
\usepackage{graphics, color}
\usepackage{xcolor}
\usepackage{mathtools}
\usepackage[margin=2.7    cm]{geometry}
\usepackage[colorlinks=true,citecolor=cyan,urlcolor=blue,linkcolor=blue]{hyperref}

\def\O{\mathrm O}

\def\R{\mathbb R}

\def\C{\mathbb C}

\def\F{\mathbb F}
\def\H{\mathbb H}
\def\V{\mathbb V}

\def\U {\mathrm{U}}

\def\N{\mathbb N}

\def\R{\mathbb R}

\newcommand{\SU}{\mathrm{SU}}
\newcommand{\SO}{\mathrm{SO}}
\newcommand{\GL}{\mathrm{GL}}
\newcommand{\Sp}{\mathrm{Sp}}

\newtheorem{theorem}{Theorem}[section]
\newtheorem{corollary}[theorem]{Corollary}
\newtheorem{lemma}[theorem]{Lemma}

\theoremstyle{definition}
\newtheorem{definition}[theorem]{Definition}
\newtheorem{example}[theorem]{Example}

\newtheorem*{ack}{Acknowledgement}
\theoremstyle{remark}
\numberwithin{equation}{section}

\newcommand{\diagentry}[1]{\mathmakebox[1.5 em]{#1}}
\newcommand{\secref}[1]{Section~\ref{#1}}
\newcommand{\thmref}[1]{Theorem~\ref{#1}}
\newcommand{\lemref}[1]{Lemma~\ref{#1}}

\newcommand{\corref}[1]{Corollary~\ref{#1}}

\begin{document}

\title[Reversibility of Hermitian Isometries]{Reversibility of Hermitian Isometries}
\author[Krishnendu Gongopadhyay  \and Tejbir Lohan]{Krishnendu Gongopadhyay \and
 Tejbir Lohan}
\address{Indian Institute of Science Education and Research (IISER) Mohali,
 Knowledge City,  Sector 81, S.A.S. Nagar 140306, Punjab, India}
\email{krishnendug@gmail.com, krishnendu@iisermohali.ac.in}
\address{Indian Institute of Science Education and Research (IISER) Mohali,
 Knowledge City,  Sector 81, S.A.S. Nagar 140306, Punjab, India}
\email{tejbirlohan70@gmail.com}
 
\date{\today}

\begin{abstract}
An element $g$ in a group $G$ is called reversible (or real) if it is conjugate to $g^{-1}$ in $G$, i.e., there exists $h$ in $G$ such that $g^{-1}=hgh^{-1}$. The element $g$  is called \emph{strongly reversible} if the conjugating element $h$ is an involution (i.e., element of order at most two) in $G$. In this paper, we classify reversible and strongly reversible elements in the isometry groups of $\F $-Hermitian spaces, where $\F=\C$ or $\H$. More precisely, we classify reversible and strongly reversible elements in the groups $\Sp(n) \ltimes \H^n$, $\U(n) \ltimes \C^n$ and $\SU(n) \ltimes \C^n$. We also give a new proof of the classification of strongly reversible elements in $\Sp(n)$. 
\end{abstract} 

\maketitle 

\section{ Introduction} \label{intro}
Let $G$ be a group. An element $g$ in $G$ is called \emph{reversible} if it is conjugate to $g^{-1}$ in $G$,  i.e., there exists $h$ in $G$ such that $g^{-1}=hgh^{-1}$. A reversible element $g$ is called \emph{strongly reversible} if  it is conjugate to $g^{-1}$  by an involution (i.e., element of order at most two) in $G$. Equivalently, $g$ is strongly reversible if it is a product of two involutions in $G$. 

 Reversible elements are also known as real elements in the literature, and strongly reversible elements are known as `strongly real' and `bireflectional',  e.g., \cite{ FZ}, \cite{KN1}, \cite{ST}, \cite{W}. A strongly reversible element is reversible,  but the converse is not true in general. It has been a problem of broad interest to investigate reversibility in groups; see the monograph \cite{FS} for an elaborate exposition of this theme. 

In \cite{S}, also see \cite[Chapter 6]{FS}, Short proved that every element in the Euclidean isometry group $\O(n) \ltimes \R^n$ is strongly reversible. Situation for the orientation-preserving isometries $\SO(n) \ltimes \R^n$ is subtle, and Short classified the strongly reversible elements in this group as well. In this paper, we aim to investigate the reversibility of elements in the isometry group of the $\F$-Hermitian space, where $\F$ is either the field of complex numbers $\C$ or the division ring of Hamilton's quaternions $\H$. Before stating the main results, we review the background material. 

We recall that every element in $\H$ can be expressed as $a=a_0 + a_1 i + a_2 j + a_3 k$, where $a_0$, $a_1$, $a_2$, $a_3$ are real numbers, and 
 $i^2=j^2=k^2=-1$,  $ij=-ji=k$,  $ jk=-kj=i$, $ki=-ik=j$. The conjugate of $a$ is given by $\bar a=a_0-a_1i -a_2j -a_3k$.   We identify the real subspace $\R \oplus \R i$ with the usual complex plane $\C$, and then one can write $\H= \C \oplus  \C j$. 

\medskip 
Let $\V$ be an $n$-dimensional right vector space over $\H$.  Let $T: \V \to \V$ be  a right linear transformation.  After choosing a suitable basis of  $\V$, we can represent $T$ by an $n \times n$ matrix over $\H$. Let $v \in \V , v \neq 0 $, $\lambda \in \H$, be such that $T(v)=v \lambda$, then for $\mu \in {\H}^{\times}$, we have 
$$T(v \mu)=(v \mu) \mu^{-1} \lambda \mu.$$Therefore, eigenvalues of $T$ occur in similarity classes and if $v$ is a $\lambda$-eigenvector, then $v \mu \in v \H$ is a $\mu^{-1} \lambda \mu$-eigenvector. Each similarity class of eigenvalues contains a unique pair of complex numbers that are complex conjugates. We often refer to them as `eigenvalues', though it should be understood that our reference is towards their similarity classes. In places where we need to distinguish between the similarity class and a representative, we shall write the similarity class of an eigenvalue representative $\lambda$ by $[\lambda]$. We shall mostly choose the complex representative of a similarity class where the argument lies in $[0, \pi]$. For an elaborate discussion on the theory of linear transformations over the quaternions, see \cite{rodman}.  

\medskip 

Let $\V=\F^{n}$ be equipped with the $\F$-Hermitian form 
\[\Phi(z,w)=\bar{z}_1w_1+\dots+\bar{z}_nw_n,\]
where $z=( z_1, \dots, z_n), \; w=( w_1, \dots, w_n)\in \F^{n}$. The group of linear transformations $g$ that preserves this form, i.e.,  $\Phi(gz, gw)=\Phi(z, w)$ for all $z, w \in \V$,  is the unitary group $\U(n, \F)$. 

In matrix notation, 
\[\U(n, \F) \coloneqq \{g \in \GL(n,\F) \mid \bar{g}^\top g =g \bar{g}^\top= I_n\},\] where $\GL(n,\F)$ is the group of invertible $n \times n$ matrices over $\F$ and $I_n$ is the identity matrix of order $n$. Note that $ \U(n, \F)$ is a compact subgroup of $\GL(n,\F)$. Following usual notation, we write $ \Sp(n)  \coloneqq  \U(n, \H)$, $ \U(n) \coloneqq  \U(n, \C) $, and $\SU(n) \coloneqq \{ g \in \U(n, \C): \det(g)=1\}$. Note that all eigenvalues of $\Sp(n)$ and $ \U(n)$ have unit modulus.
 We denote $g^{\bigstar}=\bar{g}^\top$ in the sequel. 
We will use the following definitions.
\begin{definition}\label{multiplicity}
Let $g\in \Sp(n)$ such that $e^{i\theta_l}$, $1 \leqslant l \leqslant m$, denote the distinct eigenvalues of $g$. The right vector space $\H^n$ has the following orthogonal decomposition into eigenspaces:
	\[ \H^n=\V_{\theta_1}\oplus \V_{\theta_2}\oplus \dots \oplus \V_{\theta_m},\]
	where $\V_{\theta_l}=\{v\in \H^n\mid gv=ve^{i\theta_l}\}$ for $1\leq l\leq m$.
	We define the {\it multiplicity} of $e^{i\theta_l}$ to be $\mathrm{dim}\;(\V_{\theta_l})$. Equivalently, it is the number of repetitions of the eigenvalue $e^{i\theta_l}$ in the diagonalization of $g$. 
\end{definition}

\begin{definition}
Let $g \in \U(n)$. The characteristic polynomial $\chi_g(x)$ of $g$ is called \emph{self-dual} if whenever $\lambda \neq \pm 1$ is a root of $\chi_g(x)$, so is $\lambda^{-1}$ with the same multiplicity.
\end{definition}
\medskip The Hermitian form $\Phi$ gives a natural metric $d(z, w)=\Phi(z-w, z-w)^{\frac{1}{2}}$ on $\V$. We call $(\V, d)$, a \emph{Hermitian space}.   The group $\U(n, \F) \ltimes \F^n$ acts isometrically on $(\V, d)$ as affine transformations: $T: z \mapsto Az+v$, where $A \in \U(n, \F)$ and $v \in \F^n$.  This action identifies the isometry group $\textnormal{Isom}(\V,d)$ with $\U(n, \F) \ltimes \F^n$.

\medskip  In this paper, we classify the reversible and strongly reversible elements in the group
 $\U(n, \F) \ltimes \F^n$. The reversibility depends upon the underlying $\F$ and the results in the respective groups are very different. We begin with the group $\Sp(n) \ltimes \H^n$. It is proved that every element in $\Sp(n) \ltimes \H^n$ is reversible. 
\begin{theorem}\label{revsp}
Let $g$ be an element of $ \Sp(n)  \ltimes \H^n$. Then $g$ is reversible in $ \Sp(n)  \ltimes \H^n$. \end{theorem}
However, not every element is strongly reversible in this group and the following theorem classifies strongly reversible elements in $\Sp(n) \ltimes \H^n$.

 \begin{theorem}\label{srasp} Let $g = (A,w )$ be an element of $\Sp(n)  \ltimes \H^n$. Then the following are equivalent. 
\begin{enumerate}
\item $g$ is strongly reversible in $\Sp(n) \ltimes \H^n$. 
\item Every eigenvalue class of $A$ is either $ \pm 1$ or of even multiplicity. 
\item $A$ is strongly reversible in $\Sp(n)$.\end{enumerate} 
\end{theorem} 

\medskip In the group $\U(n) \ltimes \C^n$,  not every element is reversible. The reversibility depends on the self-duality of the linear part and we prove the following.  
\begin{theorem}\label{srau}
Let $g=(A, w)$ be an element of $\U(n) \ltimes \C^n$. Then the following are equivalent. 
\begin{enumerate} 
\item $g$ is reversible in $\U(n) \ltimes \C^n$. 
\item $g$ is strongly reversible in $\U(n) \ltimes \C^n$. 
\item $A$ is strongly reversible in $\U(n)$. 
\item The characteristic polynomial of $A$ is self-dual.  
 \end{enumerate} 
\end{theorem}
We also have the following classification of reversible elements in $\SU(n) \ltimes \C^n$. 
\begin{theorem}\label{rsu} 
Let $g =(A,w )$ be an element of  $\SU(n) \ltimes \C^n$, where $A \neq I_n$. Then the following are equivalent.

\begin{enumerate}
\item $g$ is reversible in $\SU(n) \ltimes \C^n$. 
\item $A$ is reversible in $\SU(n)$.  
\item The characteristic polynomial of $A$ is self-dual. 
\end{enumerate}

 When $A=I_n$ and $n \geq 2$, the element $g=(I_n, w)$ is strongly reversible. 
\end{theorem}
From the proof of the above theorem,  the following corollary follows immediately. 
  \begin{corollary}\label{srsu-0}
Let $g=(A, w)$ in $\SU(n) \ltimes \C^n$. Suppose $A$ has an eigenvalue $-1$. 
Then the following are equivalent. 
\begin{enumerate} 
\item $g$ is strongly reversible in $\SU(n) \ltimes \C^n$.  
\item The characteristic polynomial of $A$ is self-dual.  
\end{enumerate} 
\end{corollary}

When $A$ has no eigenvalue $-1$, the situation is more subtle  in $\SU(n) \ltimes \C^n$ and $g$ may not be strongly reversible even if $A$ is so. We classify all such elements in $\SU(n) \ltimes \C^n$ which are not strongly reversible even if the linear part is strongly reversible. 

  \begin{theorem} \label{nsrsu} Let $g \in \SU(n)\ltimes \C^n$ be such that linear part of $g$ is strongly reversible. Then $g$ is not strongly reversible if and only if up to conjugacy $g$ is of the form $g = (A,v )$ such that $A= \textnormal{diag}(e^{i\theta_1}, e^{-i\theta_1},\dots, e^{i\theta_{2r}}, e^{-i\theta_{2r}}, 1)$ and 
 $v = [0,0, \dots, 0,v_1 ]$, where $v_1 \neq 0$, $r \geq 0$, and $ \theta_k  \in  (0 ,\pi), ~ k=1, \dots, r$. 
\end{theorem}

The classification of strongly reversible elements in $\SU(n) \ltimes \C^n$ follows from the above theorem. 
\begin{corollary} 
An element $g$ in $\SU(n) \ltimes \C^n$ is strongly reversible if and only if the linear part of $g$ is strongly reversible in $\SU(n)$ and $g$ does not belong to the family given in \thmref{nsrsu}. 
\end{corollary}

\medskip Note that the classification of the reversible and the strongly reversible elements in $\U(n, \F)$ is intimately related to the classification of reversibility in $\U(n, \F) \ltimes \F^n$. The complete classifications of reversibility in $\U(n)$ and $\SU(n)$ are well-known, see \cite{FS}, \cite{GP}. However, a complete classification of the strongly reversible elements in $\Sp(n)$ was not known until recently. It was raised as an open problem in \cite{FS} to classify the strongly reversible elements in $\Sp(n)$, see \cite[p. 91]{FS}. Bhunia and Gongopadhyay have given a solution to this problem in \cite[Theorem 1.2]{BG}. However, the proof in \cite{BG} is geometric and uses the notion of `projective points'. In this paper, we revisit this problem and have given a different proof to the classification of strongly reversible elements in $\Sp(n)$. Our proof uses only simple quaternionic linear algebra and may be of independent interest.

\subsection{Structure of the paper} We classify the strongly reversible elements in $\Sp(n)$ in \secref{spn}. In \secref{lemmas}, we note down a few lemmas related to the conjugacy in $\U(n, \F) \ltimes \F^n$. These lemmas are used in the proofs of the main theorems. In \secref{Reasp}, we prove \thmref{revsp} and \thmref{srasp}.  In \secref{Reasu},  we prove \thmref{srau}.   We investigate reversibility in $\SU(n) \ltimes \C^n$ in \secref{Reasu1}, \thmref{rsu} and  \thmref{nsrsu} are proven in this section. 

\medskip 
\textbf{Notation:} Let $G$ denote one of the groups $\Sp(n) \ltimes \H^n$,  $\U(n) \ltimes \C^n$, or $\SU(n) \ltimes \C^n$.  Let $L(G)$ denote the linear part of $G$. The identity element in  the group $L(G)$ is denoted by $I_n$.  The identity of $G$ is denoted by $I$. The zero element of $\F^n$ is denoted by $\mathbf{0}$.

 \section{Reversibility in $\Sp(n)$}\label{spn}
 
The classification of reversible and strongly reversible elements in $\Sp(n)$ has been obtained recently in \cite{BG} where it has been proved that every element in $\Sp(n)$ is reversible.  The following theorem to classify the strongly reversible elements was also obtained in that paper. 

\begin{theorem}\label{t1.1}{\cite[Theorem 1.2 ]{BG} }
An element ${g}$ in $\Sp(n)$ is strongly reversible if and only if every eigenvalue class of ${g}$ is either $\pm 1 $ or of even multiplicity.
\end{theorem} 

In this section, we give a different and simpler proof of the above theorem. To begin with, note the following well-known result. 
 \begin{lemma} \label{l1.1} {\cite[Theorem  5.3.6. (e)]{rodman} }
  If ${A} \in \Sp(n)$, then there exists ${U} \in \Sp(n)$ such that ${UA}{U}^{-1} = \textnormal{diag} (e^{ i\theta_{1} },e^{ i\theta_{2} } , \dots, e^{ i\theta_{n} } )$, where $ \theta_{s} \in [0,\pi]$ for all $ s \in \{1,2, \dots, n \}$.  \end{lemma} 
  
 The lemma  above also holds for $\U(n)$ and $\SU(n)$, see {\cite[Corollary 32.9]{Curtis}}. Next, we observe the following about quaternions. 
 \begin{lemma} \label{l1.2}  Let $ \alpha \in (0,\pi )$ and $\beta \in [0,\pi ]$. Let $x \in \H$ be such that $x e^{i \alpha}= e^{i \beta } x$. Then the following holds:
 \begin{enumerate}
 \item  $x=0$ if $ \alpha  \neq  \pm \beta $;
 \item $x \in \C$ if $ \alpha  = \beta $;
 \item $x \in \C j$ if $ \alpha  =-  \beta $.

 \end{enumerate}
\end{lemma} 
    
 \begin{proof}
 Let $ x = u + v j \in \H$ be such that $x e^{i \alpha}= e^{i \beta } x, \hbox{ where } u, v \in \C$. 
 This implies \begin{equation}\label{eq2.1}
   (u + v j) e^{i \alpha}= e^{i \beta } (u + v j).\end{equation}Using $  jw =  \bar{w} j $ for all $ w \in \C $ and on comparing both sides of the Equation \eqref{eq2.1}, we get the following:
 \begin{equation}\label{eq2.2}
 u  e^{i \alpha} = u e^{i \beta } \hbox{ and }  v e^{-i \alpha} = v e^{i \beta}.
 \end{equation}

The proof of the lemma now follows from Equation \eqref{eq2.2}.   \end{proof}

 \begin{lemma} \label{l1.3} Let ${A} \in  
\Sp(n)$ be such that the eigenvalue classes of ${A }$ are either $ \pm 1 $ or have even multiplicities. Then ${A}$ is strongly reversible in $\Sp(n)$. \end{lemma} 
\begin{proof} 
 Let ${A} \in \Sp(n)$. So by using \lemref{l1.1}, up to conjugacy in $\Sp(n)$, we can assume that: 
\begin{equation}\label{elem-Sp(n)}
A= e^{ i\theta_{1} } {I}_{2k_1}  \oplus  \ e^{ i\theta_{2} } {I}_{2k_2}  \oplus \dots \oplus e^{ i\theta_{r} } {I}_{2k_r} \oplus -  {I}_{s } \oplus  {I}_{t }, \end{equation} where $r, s,t \in  \N \cup \{0\} $, and if $r \neq 0$ then $ \theta_{m} \in (0 ,\pi )$ for all $ m \in \{ 1,2 , \dots, r \}$, and ${2k_1} + {2k_2} + \dots + {2k_r}  + s + t = n$.
 
 Let $L = \begin{pmatrix}
0 & j\\
 -j & 0 \\ 
 \end{pmatrix}$ in  $\Sp(2) $. Then $L^2 = I_{2}$ in $\Sp(2)$. Consider 
\begin{equation} \label{b}  {B} =  \begin{pmatrix} \diagentry{L} \\
&\diagentry{L}\\
&&\diagentry{\ddots}\\
&&&\diagentry{L}\\
&&&&\diagentry{I_{s+t}} \end{pmatrix}\end{equation} 
in $\Sp(n)$. Then
${B} {A}{B}^{-1} = {A}^{-1}$ and ${B}^2 = {I}_{n}$. Hence ${A}$ is strongly reversible in $\Sp(n)$.
\end{proof}

\begin{lemma}\label{l1.4} 
  Suppose ${A} \in \Sp(n)$ has an eigenvalue class $[\lambda]$ of odd multiplicity $k$ such that $ \lambda  \neq  \pm 1$. Define $\mathcal{R}({A})= \{ {B} \in \GL(n,\H) : {BAB}^{-1} = {A}^{-1} \}$. Then up to conjugacy, an element ${B}$ in $\mathcal{R}({A})$ is of the form  
$$ \begin{pmatrix}
 B_{1}  j & \\
 &  B_{2} \\ 
 \end{pmatrix},$$
where $B_{1}  \in  \GL(k,\C), B_{2} \in \GL(n-k,\H)$.
\end{lemma}

\begin{proof}  Let $\theta_{1} \in (0 ,\pi  )$ be such that $e^ {  i\theta_{1} } $ is a chosen representative of the eigenvalue class $[\lambda]$ of odd multiplicity $k$.  Let us first consider the case when $[\lambda]$ is the unique eigenvalue class of $A$. Up to conjugacy in $\Sp(n)$, we can assume that : 
$$A=[ a_{s,t}]_{1\leq s,t \leq n}  = e^{i\theta_{1} } {I_{n}}.$$
 Let ${B} = [b_{s,t}]_{1\leq s,t \leq n} \in \mathcal{R}({A})$.  Now, $ {BA}{B}^{ -1} = { A}^{ -1}$ implies $( b_{s,t})e^{i\theta_{1} } =  e^{- i\theta_{1} } ( b_{s,t}) ~ \text{if  } 1\leq s,t \leq n$.  So by using \lemref{l1.2}, we get  $( b_{s,t}) =  ( w_{s,t}) j \hbox{ for  some} \ w_{s,t} \in \C $. Therefore, we get $B = B_1 j $, where $B_{1}  = [ w_{s,t}]_{1\leq s,t \leq n} \in \GL(n,\C )$.

 Now, consider the case when $A$ has at least two distinct eigenvalue classes. Up to conjugacy in $\Sp(n)$, we can assume that: 
$$A=[ a_{s,t}]_{1\leq s,t \leq n}  = e^{i\theta_{1} } {I_{k}} \oplus D, \hbox{where} \ D= \textnormal{diag} (e^ {i\theta_{k+1}}, e^ {i\theta_{k+2}},\dots, e^ {i\theta_n}),$$ where $\theta_{s} \in [0,\pi ]$, $\theta_{1} \neq  \theta_{s}$ for all $k +1\leq s \leq n$.  Note that $a_{s,t}=0$ for all $s \neq t$.

Let ${B} = [b_{s,t}]_{1\leq s,t \leq n} \in \mathcal{R}({A})$. From $ {BA}{B}^{ -1} = { A}^{ -1}$, we get  
$${BA}  = { A}^{ -1} {B} 
    \Longleftrightarrow  ( b_{s,t})( a_{t,t}) = (a_{s,s}^{-1}) ( b_{s,t}), \hbox{for all }  1\leq s,t \leq n.$$
This implies 
     $$( b_{s,t})e^{i\theta_{1} } =  e^{- i\theta_{1} } ( b_{s,t}) ~ \text{if  } 1\leq s,t \leq k;$$
     $$( b_{s,t}) e^{i\theta_{t} }   = e^{- i\theta_{1} } ( b_{s,t}) ~ \text{if  } 1\leq s \leq k, k +1\leq t \leq n;$$
   $$( b_{s,t}) e^{i\theta_{1} }   = e^{- i\theta_{s} } ( b_{s,t}) ~\text{if  }  k +1\leq s \leq n, 1\leq t \leq k.$$

As we have chosen $\theta_{1} \in (0,\pi ),  ~\theta_{s} \in [0,\pi ]$ such that $\theta_{1} \neq  \theta_{s}$ for all $k +1\leq s \leq n, s \in \N$. Therefore, by using \lemref{l1.2}, we get the following:

    \medskip 
$\begin{cases}
     ( b_{s,t}) = ( w_{s,t}) j  \hbox{ for  some}  \ w_{s,t} \in \C   & \text{if  $ 1\leq s,t \leq k ;$}\\
      b_{s,t} =0 & \text{if  $ 1\leq s \leq k, k +1\leq t \leq n ;$ } \\
    b_{s,t} =0  & \text{if   $ k +1\leq s \leq n , 1\leq t \leq k $}.
    \end{cases}$

   \medskip 
This implies that the matrix $B$ is of the form 
  $$ {B} =
 \begin{pmatrix}
 B_{1}  j & \\
 &  B_{2} \\ 
 \end{pmatrix},$$ 
where $ B_{1}  = [ w_{s,t}]_{1\leq s,t \leq k} \in  \GL(k,\C )$  and $ B_{2} \in  \GL(n-k,\H )$. 

This proves the lemma.
\end{proof}

\begin{lemma}\label{l1.5} 
  If ${A} \in \Sp(n)$ has an eigenvalue class $\lambda$,  $\lambda \neq \pm 1$, of odd multiplicity, then ${A}$ is not strongly reversible in $\Sp(n)$.
   \end{lemma}
 
\begin{proof} Assume that ${A}$ is strongly reversible in $\Sp(n)$. Then there exists $B$ in $\Sp(n)$ such that
${BA}{B}^{ -1} = { A}^{ -1}$, ${B}^2 =  {I_{n}}$, and ${B}{B}^{\bigstar} =  {I_{n}}$. Since $B \in \mathcal{R}({A})$, so by using \lemref{l1.4}, we can write $B$ in the following form: 
  $$ {B} =
 \begin{pmatrix}
 B_{1}  j& \\
 &  B_{2} \\ 
 \end{pmatrix},  $$ 
 where $B_{1}  \in  \GL(k,\C), B_{2} \in \GL(n-k,\H)$.

Now, we see that the relations 
 ${B}^{2} = {I_{n}}$ and ${B}{B}^{\bigstar} = {I_{n}}$ implies
 $$(B_{1}   j ) ( B_{1}  j) =  {I_{k}} \ \hbox{and} \ (B_{1}  j )  ( B_{1}  j) ^{\bigstar} =  {I_{k}}.$$
So by using $wj = j\bar{w}$ for all $w \in \C$, we get $(B_{1}) ( \overline{ B_{1}} ) = - {I_{k}}$ and $(B_{1}) ( B_{1}^{\bigstar}) = {I_{k}}$. In particular,  $B_{1}  \in   \GL(k,\C )$ such that $B_{1}  ^{-1} = - \overline{B_{1}} =  B_{1}^{\bigstar}$. 
This implies $B_{1}^{\top} = - B_{1}$. Hence $ \det(B_{1}^{\top})= \det(- B_{1}), \hbox{ i.e., } \det( B_{1}) = (-1)^{k} \det( B_{1}) = - \det( B_{1})$, where $k$ is odd. 

Therefore,  $\det( B_{1}) =0$. Hence $B_1$ is not invertible. This is a contradiction. So our initial assumption that $A$ is strongly reversible can not be true. 
This proves the lemma. \end{proof} 

\subsection{Proof of \thmref{t1.1}}
The proof follows from \lemref{l1.3} and \lemref{l1.5}.
\hfill $\square$

\bigskip 
\subsection{On reversibility in $\SU(n)$}  

First, we shall note the following lemma that will be used in the proof of Theorem \ref{nsrsu}.

\begin{lemma} \label{le2.3}  Suppose $A  \in \SU(2n)$ has a self-dual characteristic polynomial such that no eigenvalue of $A$ is equal to $1$ or $-1$. Let $B \in \GL(2n,  \C )$ such that  $BAB^{-1} = A^{-1}$ and $B^2 = I_{2n}$. 
Then $\det (B ) = (-1) ^n $. 
  \end{lemma} 
\begin{proof} We will prove this lemma by using strong induction.
First consider $n = 1$. Then up to conjugacy in $\SU(2)$, we can write A as:

 $A  =\begin{pmatrix} e^ {i\theta} &  \\  & e^ {- i\theta}  \end{pmatrix}$, where $\theta \in (0,\pi )$. 
   Let $B = [b_{p,q}]_{1\leqslant p,q,\leqslant 2} \in  \GL(2,  \C )$ be such that  $BAB^{-1} = A^{-1}$ and $B^2 = I_{2}$. 
   Now, $BAB^{-1} = A^{-1}$ implies 
 $$  \begin{pmatrix}
e^{i\theta} b_{1,1}  & e^{-i\theta}  b_{1,2}   \\
e^{i\theta}b_{2,1}  &e^{-i\theta}  b_{2,2}    \\
 \end{pmatrix}\ =  \begin{pmatrix}
e^{-i\theta} b_{1,1}  & e^{-i\theta}  b_{1,2}   \\
e^{i\theta}b_{2,1}  & e^{i\theta} b_{2,2}  
 \end{pmatrix}.$$
Thus, $b_{1,1} = b_{2,2}  = 0  $ and $b_{1,2}$,  $b_{2,1} $ are non-zero complex numbers. Now, $B^2=I_2$ implies $ b_{1,2} b_{2,1} = 1$. So $\det(B) = - b_{1,2} b_{2,1} = -1 = (-1)^{1}$. Thus lemma holds for $n=1$. Next, assume that lemma holds for all $n\leqslant k$. Now by using induction hypothesis, we will show that lemma holds for $n=k+1$.

Consider $n=k+1$. Let $ e^{i\theta_1}$ is an eigenvalue of $A$ with multiplicity $r$, where $1 \leqslant r \leqslant n$ and $\theta_1 \in (0,\pi )$. 
 Now there are two possible cases:
 \begin{enumerate}

 \item  Suppose that $n=r$. Then upto conjugacy in $\SU(2n)$, we can write $A$ as: ${ A} = \textnormal{diag} (\underset{n\text{-times}}{\underbrace{e^{i\theta_1}, \dots, e^{i\theta_1}}},\underset{n\text{-times}}{\underbrace{e^ {-i\theta_1}, \dots,e^{-i\theta_1}}})$.
Let $B \in \GL(2n,  \C )$ be such that $BAB^{-1} = A^{-1}$ and $B^2 = I_{2n}$.  Then $ {B} =
 \begin{pmatrix}
 & P   \\
   Q & \\ \end{pmatrix}$, where $P,Q \in \GL(n,  \C )$ such that $PQ =I_n$.  Note that  ${B} =
 \begin{pmatrix}
 P  & \\
 &  Q \\ 
 \end{pmatrix} \begin{pmatrix}
 & I_n   \\
   I_n & \\ \end{pmatrix}$. This implies that $\det(B)=\det(P)\det(Q)(-1)^n$.  But $PQ =I_n$. So $ \det(PQ)= \det(P)\det(Q)= det(I_n)=1$. Thus we get $\det(B)= (-1)^n$.
  
 \item  Suppose that $r<n$. Then upto conjugacy in $\SU(2n)$, we can write $A$ as:  
$ {A} =
 \begin{pmatrix}
 A_{1}  & \\
 &  A_{2} \\ 
 \end{pmatrix}$ such that ${ A_1} = \textnormal{diag} (\underset{r\text{-times}}{\underbrace{e^{i\theta_1}, \dots, e^{i\theta_1}}},\underset{r\text{-times}}{\underbrace{e^ {-i\theta_1}, \dots,e^{-i\theta_1}}})\in \SU(2r)$ and ${ A_2} = \textnormal{diag} (e^{i\theta_{r+1}},  e^{-i\theta_{r+1}}, \dots, e^ {i\theta_{n}}, e^{-i\theta_{n}})\in \SU(2(n-r))$, where $\theta_1 \neq \theta_\ell$ and $ \theta_\ell  \in  (0 ,\pi)$, for all $ k = r+1,r+2 , \dots,n$. 
 
 Let $B \in \GL(2n,  \C )$ be such that  $BAB^{-1} = A^{-1}$ and $B^2 = I_{2n}$.  Then $ {B} =
 \begin{pmatrix}
 B_{1}  & \\
 &  B_{2} \\ 
 \end{pmatrix}$, where $B_1$ and $B_2$ satisfy following conditions:
 \begin{itemize}
 \item $B_1 \in \GL(2r,  \C )$ such that  $B_1 A_1 B_1^{-1} = A_1^{-1}$ and $B_1^2 = I_{2r}$. 
 \item $B_2 \in \GL(2(n-r),  \C )$ such that  $B_2 A_2 B_2^{-1} = A_2^{-1}$ and $B_2^2 = I_{2(n-r)}$. 
  \end{itemize}

 Since $r<n=k+1$ and $(n-r) < n=k+1$, so by induction hypothesis, $\det(B_1)= (-1)^r $ and $\det(B_2)= (-1)^{n-r}$. This implies $\det(B)= \det(B_1) \det(B_2)= (-1)^r(-1)^{n-r}= (-1)^{n}$. 
 \end{enumerate}
 Therefore, from both the cases we get lemma holds for $n=k+1$. Hence, by induction principle, we conclude that lemma holds for each $n\in \N$. 
\end{proof} 

Now we further note the following result from \cite{GP} that classifies strongly reversible elements in $\SU(n)$. 

\begin{lemma}\label{lem-strong-SU(n)} \cite[Proposition 3.3]{GP} Suppose $ A \in \SU(n)$ has a self-dual characteristic polynomial. Then $A$ is not strongly reversible in $\SU(n) $ if and only if  $ n \equiv 2 \mod 4$ and no eigenvalue of $A$ is equal to $1$ or $-1$. 
\end{lemma}

Note that if part of above lemma follows directly from the Lemma \ref{le2.3}. For the converse part see \cite[p. 74]{FS}.

 \section{Conjugacy in $\U(n, \F) \ltimes \F^n$} \label{lemmas}
 
Let $G$ denote one of the groups $\Sp(n) \ltimes \H^n$,  $\U(n) \ltimes \C^n$, or $\SU(n) \ltimes \C^n$. 
\begin{lemma} \label{l2.1} 
Every element $g$ in $G$,  upto conjugacy, can be written as $g = ({A},v)$ such that ${A}(v) =v$ and  ${ A} = \textnormal{diag} (e^{i\theta_1} , e^{i\theta_2},\dots,e^{i\theta_r}, - I_s , I_t)$,  where $r,s,t \in  \N \cup \{0\}$,$\theta_k  \in  (0 ,\pi)$,  $ k = 1,2 , \dots, r$ and $v$ is of the form $v = [0,0, \dots, 0,v_{1}, v_{2},\dots,v_{t} ]$.  Further, If $1$ is not an eigenvalue of the linear part of $g$, then upto conjugacy $g$ is of the form $g= (A,\mathbf{0})$.
\end{lemma} 
\begin{proof} 
\medskip Let $g \in G$ be an arbitrary element.  After conjugating $g$ by a suitable element $(B,\mathbf{0})$ in $G$, we can assume that 
$g = (A ,w)$ such that $A= \textnormal{diag} (e^ {i\theta_1}, e^ {i\theta_2},\dots, e^ {i\theta_r}, -I_s, I_t)$, where $\theta_k  \in  (0 ,\pi)$,  $k = 1,2 , \dots, r$ and $r+s+t = n$. Now there are two possible cases:
\begin{enumerate}

\item Suppose $1$ is not an eigenvalue of $A$. 
   So the linear transformation ${A}-  {I_n}$ is invertible. Therefore, we can choose $x_o =({A } -  {I_n}  )^{-1} (w)  \in \F^n$. 
   
   Consider $h = ({I_n}, x_o) \in G$, i.e., $h(x) = x + x_o $ for all $x \in \F^n$. Note that for all $x \in \F^n$,
 $$ hgh^{-1} (x) = hg(x- x_o) = h({A}x - {A}x_o + w ) ={A}x + w - ({A}-I_n)x_o.$$
 This implies $hgh^{-1} (x)= {A}(x) + \mathbf{0} $ for all $x \in \F^n$ since $x_o =({A} -  {I_n} )^{-1} (w)$. By taking  $v = \mathbf{0}$, we have $hgh^{-1} =({A} , \mathbf{0})\hbox{ such that } {A}(\mathbf{0}) =  \mathbf{0}$. 
  
 \item Let $1$ be an eigenvalue  of ${A}$. In this case $t >0$ and 
${A} - {I_n}$ has rank $r+s = n-t  < n$. So we can choose an element $u \in \F^n$ having the last $n -(r+s)$ coordinates zero such that $ [({A} - {I_n}  )(u)]_\ell = w_\ell \hbox{ for all }  \ 1  \leq \ell \leq r+s, \hbox{where } w = [w_\ell]_{1  \le \ell \le n}$. Let $v = w{ - }({A} {-}  {I_n} )(u)$. 
  Then $v = [0,0, \dots, 0,w_{r+s+1}, w_{r+s+2},\dots,w_n ] \hbox{ and} \ {A}(v) =v $.
  
  Consider $h = (  {I_n},u) \in G, \hbox{ i.e., }h(x) = x+u$ for all $x \in \F^n$. Note that for all $x \in \F^n$, 
$$hgh^{-1}(x) = hg(x-u) = h({A}x - {A}u +w) = {A}x + w-({A} - {I_n})(u) = {A}x +v.$$ 
\end{enumerate}
This proves the lemma.
\end{proof}

In the following lemma, we get a sufficient condition for reversibility of $g \in G$.
\begin{lemma}\label{l2.2}  Let $g = ({A},v)$ in $G$ be such that ${A}(v) = v$. If there exists an element  ${B}$  in $L(G)$ such that  ${BA}{B}^{-1} ={ A}^{-1} $ and ${B}(v)  = - v$, then $g$ is {reversible} in $G$.\end{lemma}
\begin{proof} 
Consider $ h = ({B }, \mathbf{0}) \in G$, where $\mathbf{0}$ is the zero element in $\F^n$. Then $hgh^{-1} = g^{-1} $. 
This proves the lemma. \end{proof} 

In the following lemma, we get a sufficient condition for strong reversibility of $g \in G$.
\begin{lemma} \label{lem-str-real} 
 Let $g = ({A},v)$ in $G$ be such that ${A}(v) = v$. If there exists an element  ${B}$ in $L(G)$ such that ${BA}{B}^{-1} ={ A}^{-1} $,  ${B}(v)  = - v$, and ${ B}^{2} =  I_{n}$, then $g$ is {strongly reversible} in $G$.  
\end{lemma}

\begin{proof} Consider $ h = ({B }, \mathbf{0}) \in G$, where $\mathbf{0}$ is the zero element in $\F^n$.
Then $h$ is an involution in $G$ such that $hgh^{-1} = g^{-1} $. This proves the lemma. 
\end{proof}

In the following lemma, we get a necessary condition on the linear part of $g \in G$ for reversibility and strongly reversibility of $g$.

\begin{lemma}\label{l2.3} The following holds.
\begin{enumerate}
\item If $g = ({A} , v )$ is reversible in $G$, then ${ A }$ is reversible in $L(G)$. 
\item If $g = ({A} , v )$ is strongly reversible in $G$, then ${ A }$ is strongly reversible in $L(G)$. 
\end{enumerate}
\end{lemma} 
\begin{proof} 
 Consider a group homomorphism $\phi : G \to L(G) $, which sends each $g \in G$ to its linear part, i.e., $\phi(A, v) = A$.  Let $g$ be reversible in $G$. Then there exists $h$ in $G$ such that $hgh^{-1} = g^{-1}$. This implies $\phi(h)\phi(g)(\phi(h))^{-1} =( \phi(g))^{-1}$. Also, if $h$ is an involution in $G$, then $\phi(h)$ will be an involution in $L(G)$. The lemma now follows from the above observations.
\end{proof} 

\section{Reversibility in $\Sp(n) \ltimes \H^n$}\label{Reasp}
\subsection{Proof of \thmref{revsp} }
Let  $g \in  \Sp(n)  \ltimes \H^n$.  By using \lemref{l2.1}, up to conjugacy, we can write $g= (A,v)$ such that ${A}(v) =v$, where $A = \textnormal{diag} (e^{i\theta_1} , e^{i\theta_2},\dots,e^{i\theta_r}, - I_s , I_t)$,  $\theta_k  \in  (0 ,\pi)$, $r,s,t \in  \N \cup \{0\}$, $r+s+t = n$, and $v = [0,0, \dots, 0,v_{1}, v_{2},\dots, v_{t} ] \in  \H^{n}$. 

If  $t =0$, i.e., $1$ is not an eigenvalue of $A$, then up to conjugacy, $g = ({A}, \mathbf{0})$, where $A \in \Sp(n)$. But every element in $\Sp(n)$ is reversible, see {\cite[Proposition 3.1]{BG}}. So $A$ is reversible. Hence $g$ is reversible in $ \Sp(n) \ltimes \H^n$.

If $t > 0$, consider ${B} = \textnormal{diag} ( j,j,\dots,j,- I_{s+t}) \in \Sp(n)$. 
 Then ${BA}{B}^{-1} ={ A}^{-1}$, ${B}(v)  = - v$. Hence proof follows from \lemref{l2.2}.   \hfill $\square$ 

\begin{lemma}\label{lt5}
Suppose that $g = (A,w )$ be an element of $\Sp(n) \ltimes \H^n$.  Let every eigenvalue class of $A$ is either $ \pm 1$ or of even multiplicity. Then $g$ is strongly reversible in $\Sp(n)  \ltimes \H^n$.\end{lemma} 
\begin{proof} Up to conjugacy in $\Sp(n)\ltimes \H^n$, we can write $g$ as $g = (A, v )$ as in the \lemref{l2.1}. Further, we can assume that $A$ has the form as given in the Equation \eqref{elem-Sp(n)}.  Now, consider
\begin{equation} {B} =  \begin{pmatrix} \diagentry{L} \\
&\diagentry{L}\\
&&\diagentry{\ddots}\\
&&&\diagentry{L}\\
&&&&\diagentry{I_{s}}\\
&&&&&\diagentry{-I_{t}}  \end{pmatrix}\end{equation} 
in $\Sp(n)$, where $L = \begin{pmatrix}
0 & j\\
 -j & 0 \\ 
 \end{pmatrix}$. 
Then ${BAB}^{-1} ={ A}^{-1}$,  ${B}(v)= - v$, and ${B}^{2} =I_n$. Hence proof follows from Lemma \ref{lem-str-real} . \end{proof}

\subsection{Proof of \thmref{srasp} } 
\medskip  The proof of the theorem follows from \thmref{t1.1}, \lemref{l2.3}, and Lemma \ref{lt5}.  \hfill $\square$

\section{Reversibility in $\U(n) \ltimes \C^n$} \label{Reasu}

\subsection{Proof of \thmref{srau}}
The implication $(2) \Rightarrow (1)$ is always true since every strongly reversible element in a group is always reversible. We shall prove the rest. 

\medskip 
$(4) \Rightarrow (2)$: Suppose $A$ has a self-dual characteristic polynomial. 

Therefore, by using \lemref{l2.1}, we can assume that $g = ({A},v)$ such that  ${A}(v) =v$ and 
 ${ A} = \textnormal{diag} (e^{i\theta_1},  e^{-i\theta_1}, \dots, e^ {i\theta_r}, e^{-i\theta_r}, - I_s , I_t)$, $\theta_k  \in  (0 ,\pi)$, $ k = 1,2 , \dots,r$ where $r,s, t \in \N \ \cup  \ \{0\} $ and $v$ is of the form $v = [0,0, \dots , 0,v_{1}, v_{2},\dots, v_{t} ]$. Let $K=\begin{pmatrix} 0 & 1 \\ 1 & 0 \end{pmatrix}$. Consider 
\begin{equation} \label{c}  {B} =  \begin{pmatrix} \diagentry{K} \\
&\diagentry{K}\\
&&\diagentry{\ddots}\\
&&&\diagentry{K}\\
&&&&\diagentry{I_{s}}\\
&&&&&\diagentry{-I_{t}}  \end{pmatrix}\end{equation} 
in $\U(n)$. Then ${BAB}^{-1} ={ A}^{-1}$,  ${B}(v)= - v$, and ${B}^{2} =I_n$. Hence proof follows from Lemma \ref{lem-str-real}. 

$(2) \Rightarrow (3)$: Follows from  \lemref{l2.3}.

$(3) \Leftrightarrow (4)$: This equivalence follows from \cite[Proposition 3.1]{GP}. 

$(1) \Rightarrow (4)$: Follows from \lemref{l2.3} and noting the fact that an element $A$ in $\U(n)$ is reversible if and only if it has a self-dual characteristic polynomial,  see {\cite[Corollary 3.2]{GP}}, {\cite[Theorem 8]{Ellers}}.  \hfill  $\square$
\section{Reversibility in $\SU(n) \ltimes \C^n$} \label{Reasu1}
\subsection{Proof of \thmref{rsu}}
$(1) \Rightarrow (2)$: Follows from the \lemref{l2.3}.

$(2) \Leftrightarrow (3)$: This equivalence follows from {\cite[Theorem 4.22]{FS}}, {\cite[Proposition 3.3]{GP}}.

$(3) \Rightarrow (1)$: 
Suppose $A$ has a self-dual characteristic polynomial. 

Therefore, by using \lemref{l2.1}, we can assume that $g = ({A},v)$ such that  ${A}(v) =v$ and 
 ${ A} = \textnormal{diag} (e^{i\theta_1},  e^{-i\theta_1}, \dots, e^ {i\theta_r}, e^{-i\theta_r}, - I_s , I_t)$, $\theta_k  \in  (0 ,\pi)$, $ k = 1,2 , \dots,r$ where $r,s, t \in \N \ \cup \ \{0\}, $ and $v$ is of the form $v = [0,0, \dots , 0,v_{1}, v_{2},\dots, v_{t} ]$. Further, note that if $A \neq I_n$, then $(r,s)  \neq (0,0)$.
 
 Suppose $A \neq I_n$. Consider $B$ as given in the Equation \eqref{c}.  Note that $\det(K)=-1$ and $\det( B)=(-1)^{r+t} $. If $\det (B)=1$, then choose this $B$. If not then  we choose $B \in \SU(n)$ in the following way. 
If $s \neq 0$, we choose $B \in \SU(n)$ as 
\begin{equation} {B} =  \begin{pmatrix} \diagentry{K} \\
&\diagentry{K}\\
&&\diagentry{\ddots}\\
&&&\diagentry{K}\\
&&&&\diagentry{-1}\\
&&&&&\diagentry{I_{s-1}}\\
&&&&&&\diagentry{-I_{t}}  \end{pmatrix}\end{equation} 
Note that this $B$ is also an involution in $ \SU(n)$.  When $r \neq 0$, the element $B$ can also be chosen as 
\begin{equation} \label{d} {B} =  \begin{pmatrix} \diagentry{J} \\
&\diagentry{K}\\
&&\diagentry{\ddots}\\
&&&\diagentry{K}\\
&&&&\diagentry{I_{s}}\\
&&&&&\diagentry{-I_{t}}  \end{pmatrix}\end{equation} 
where $J=\begin{pmatrix} 0 & -1\\1& 0 \end{pmatrix}$. Then $\det (B)=1$, though note that $B$ is no more an involution. 

 Thus, we have constructed $B$ in $\SU(n)$ such that  ${BAB}^{-1} ={ A}^{-1}$,  ${B}(v)= - v$. The theorem now follows from Lemma \ref{l2.2} and \corref{cor-id-SU(n)} proven below. 
\hfill $\square$ 

\medskip The following lemma follows by combining \corref{srsu-0}, Lemma \ref{lem-strong-SU(n)} and \lemref{l2.1}.  
  \begin{lemma}\label{srsu-1}
Let $g=(A, w)$ in $\SU(n) \ltimes \C^n$. Assume that either one of the following holds: 

\begin{itemize} 
\item[(i)] $A$ has an eigenvalue $-1$.

\item[(ii)] $n \equiv 0 \mod 4$ and $A$ does not have any eigenvalue equal to $1$ or $-1$. 
\end{itemize}
Then the following are equivalent. 
\begin{enumerate} 
\item $g$ is strongly reversible in $\SU(n) \ltimes \C^n$.  
\item $A$ is strongly reversible in $\SU(n)$. 
\item The characteristic polynomial of $A$ is self-dual.   \end{enumerate} 
\end{lemma}

\medskip Note that if $A$ has no eigenvalue $-1$, then $g$ may not be strongly reversible even if $A$ is so. The following example demonstrates this. 
\begin{example} \label{ex1} 
 Consider $g = (A ,v ) \in \SU(5)\ltimes \C^5$ such that  $A = \textnormal{diag} (e^{i\theta},e^{-i\theta},e^{i\phi},e^{-i\phi},1 )$ and $v = (0,0,0,0,1)$, where $ \theta, \phi \in (0,\pi)$ such that $ \theta \neq \phi$. Note that $A$ is strongly reversible in $\SU(5)$. 
 
 Assume that $g$ is strongly reversible in $ \SU(5)\ltimes \C^5$. Then there exists $ h = (B ,u )$ in $ \SU(5)\ltimes \C^5$ such that 
$ hgh^{-1} =g^{-1} $ and $h^2 = I$, where  ${I}$ is the identity in $\SU(5)\ltimes \C^5.$ Using $hgh^{-1} =g^{-1} $ and $ h^2 ={I}$, we have $B \in \SU(5)$ such that for all  $x \in \C^5$, 
\begin{equation}\label{eqn00} BAB^{-1}(x) -  BAB^{-1} (u) + B(v) + u = A^{-1}(x) - A^{-1}(v), \end{equation}
\begin{equation} \label{eqn01} 
B^2(x) + B(u) + u = I_{5}(x). \end{equation} 
This implies
\begin{equation}  \label{eqn02}     BAB^{-1} =A^{-1},  ( A^{-1} - I_5)(v-u) = - (B + I_5)(v), \end{equation} 
 \begin{equation}  \label{eqn03} \  B ^2 = I_5 , (B + I_5 )(u) =0. \end{equation} 

 From the equation $ BAB^{-1} =A^{-1}$, we see that,
$$ \begin{pmatrix}
e^{i\theta} b_{1,1}  & e^{-i\theta} b_{1,2}  & e^{i\phi}b_{1,3}  & e^{-i\phi}b_{1,4} & b_{1,5} \\ 
e^{i\theta} b_{2,1}  &e^{-i\theta}  b_{2,2}  & e^{i\phi}b_{2,3}  & e^{-i\phi} b_{2,4} & b_{2,5} \\
e^{i\theta}  b_{3,1}  & e^{-i\theta}  b_{3,2}  & e^{i\phi}b_{3,3}  & e^{-i\phi}b_{3,4} & b_{3,5}  \\
e^{i\theta}  b_{4,1}  & e^{-i\theta} b_{4,2}  & e^{i\phi} b_{4,3}  & e^{-i\phi} b_{4,4} & b_{4,5} \\
e^{i\theta}  b_{5,1}  & e^{-i\theta} b_{5,2}  & e^{i\phi}b_{5,3}  & e^{-i\phi}b_{5,4} & b_{5,5} \\
 \end{pmatrix} $$
 
 $$=  \begin{pmatrix}
e^{-i\theta} b_{1,1}  & e^{-i\theta} b_{1,2}  & e^{-i\theta} b_{1,3}  & e^{-i\theta} b_{1,4} &  e^{-i\theta} b_{1,5}\\
e^{i\theta}b_{2,1}  &e^{i\theta} b_{2,2}  & e^{i\theta} b_{2,3}  & e^{i\theta} b_{2,4} & e^{i\theta} b_{2,5} \\
e^{-i\phi}b_{3,1}  & e^{-i\phi}  b_{3,2}  & e^{-i\phi} b_{3,3}  & e^{-i\phi} b_{3,4} & e^{-i\phi} b_{3,5}  \\
e^{i\phi}   b_{4,1}  &e^{i\phi}  b_{4,2}  & e^{i\phi} b_{4,3}  &e^{i\phi}  b_{4,4} & e^{i\phi} b_{4,5} \\
 b_{5,1}  &  b_{5,2}  &  b_{5,3}  & b_{5,4} & b_{5,5} \\ \end{pmatrix}. $$
 
 \vspace{5mm}
As $ \theta, \phi \in (0,\pi)$ such that $ \theta \neq \pm \phi$, so from the above matrix equation,  we get that matrix $B$ has the following block diagonal form:

\begin{equation}  \label{eqnb1} B = \begin{pmatrix}
0 & a & &  &  \\
b & 0 & &  & \\ 
&  & 0 & c & \\
 &  & d & 0 & \\
&  &   &  & \alpha \\ 
\end{pmatrix},  \hbox{ where } a, b,c,d,\alpha \in \C \setminus \{0\}.\end{equation}

Using $( A^{-1} - I_5)(v-u) = - (B + I_5)(v)$ with Equation \eqref{eqnb1} and noting that $v=(0,0,0,0,1)$, we obtain $\alpha=-1$. Now, $ B ^2 = I_5 $ and Equation \eqref{eqnb1} implies $ab= cd=1$. Thus, we have 
$$B = \begin{pmatrix}
0 & a & &  &  \\
b & 0 & &  & \\ 
&  & 0 & c & \\
 &  & d & 0 & \\
&  &   &  & -1 \\ 
\end{pmatrix}, \hbox{ where } a, b,c,d\in \C  \hbox{ such that  } ab = cd=1.$$
 Thus, $ \det (B ) = -abcd= -1$. This is a contradiction since $B \in \SU(5)$. Hence $g$ is not strongly reversible in $\SU(5) \ltimes \C^5$. 
\end{example} 

\medskip  The above example demonstrates the need for classifying those elements $(A, v)$ in $\SU(n) \ltimes \C^n$ which are strongly reversible when $A$ is strongly reversible. The following lemma is proven in this direction. 

\begin{lemma}\label{lem-srsu-2}
 Let $g = ({A},v)\in  \SU(n) \ltimes \C^n$ such that ${ A} = \textnormal{diag} (e^{i\theta_1},  e^{-i\theta_1}, \dots, e^ {i\theta_r}, e^{-i\theta_r},I_t)$ and $v = [0,0, \dots , 0,v_{1}, v_{2},\dots, v_t]$, where $t = n-2r$, $\theta_k \in (0,\pi)$, $k =1,2, \dots,r$ and $r,t \in \N$. 
 Assume that one of the following holds: 
\begin{enumerate} 
\item $r$ and $t$ are both even or both odd.
\item $v_{m} =0$ for some $m \in \{1,2 , \dots,t\}$.
\item $v_{m} \neq 0$ for all $m \in \{1,2 , \dots,t\}$ and one of the following holds:
\begin{itemize}
\item[(\textit{i})] $r$ is odd and $t$ is even.
\item[(\textit{ii})]$r$ is even and $t$ is odd, $t \neq 1$.
\end{itemize}
\end{enumerate} 
Then $g$ is strongly reversible in $\SU(n) \ltimes \C^n$.  
\end{lemma}
\begin{proof}
To prove this lemma, it is sufficient to find an involution $B$ in $\SU(n)$ such that ${BA}{B}^{-1} ={ A}^{-1} $, $B^2= I_n$ and ${B}(v)= - v$. Then proof will follow from the Lemma \ref{lem-str-real}. 

Let $B_{r} = \begin{pmatrix} \diagentry{K} \\
&\diagentry{K}\\
&&\diagentry{\ddots}\\
&&&\diagentry{K}\\
&&&&\diagentry{K}  \end{pmatrix}  \in \GL(2r,\C)$.
    
    Then $\det(B_r)= (-1)^r$. 

Note that we can construct the desired involution $B \in\SU(n)$ in the following way :
\begin{enumerate} 
\item Consider 
\begin{equation} \label{e} {B} = B_{r} \oplus (-I_t). \end{equation}
Then $\det(B)= (-1)^r (-1)^{t}= (-1)^{r+t}$. Therefore, if $r$ and $t$ are such that either both are even or both are odd, then $r+t$ will be even. Hence $\det(B)=1$.
\medskip 
\item  If $v=\mathbf{0}$, then $g$ is conjugate to $(A, \mathbf{0})$ and hence strongly reversible. So without loss of generality, we can assume that $v\neq \mathbf{0}$ such that $v_{t} = 0$. Consider $B$ as given in the Equation \eqref{e}. If $\det(B) =1$, then choose this $B$.  Otherwise, replace the $n$th entry of the $n$th row  (i.e., the row corresponding to the zero entry $v_t$ in $v$) of $B$ by $1$. That would make $\det(B)=1$ and $B$ would still satisfy the desired conditions making $g$ strongly reversible. 
\medskip 
\item Let $P= \begin{pmatrix}0& -\frac{v_{1}}{v_{2}} \\ -\frac{v_{2}}{v_{1}} & 0 \end{pmatrix}$ and $Q= \begin{pmatrix} 0 & 0& -\frac{v_{1}}{v_{3}} \\ 0 &-1 & 0\\
-\frac{v_{3}}{v_{1}} &0 & 0
 \end{pmatrix}$. 
 
 \vspace{5mm}
 Then $\det(P)= -1$ and $\det(Q)= 1$. Therefore, by using $P$ and $Q$, we can consider the following involution $B \in\SU(n)$:
 
In the subcase $(i)$, consider
$$B= B_r \oplus P \oplus (-I_{t-2}),$$
 where $t \in \N$ such that $t\geqslant2$, and in the subcase $ (ii)$, consider 
$$B= B_r \oplus Q \oplus (-I_{t-3}),$$ 
where $t \in \N$ such that $t\geqslant3$.
\end{enumerate} 

Thus, in each of the above cases, we have constructed an {\it involution} $B$ in $\SU(n)$ such that ${BA}{B}^{-1} ={ A}^{-1} $ and ${B}(v)  = - v$. Hence proof follows from Lemma \ref{lem-str-real}. 
\end{proof}
\begin{corollary}\label{cor-id-SU(n)}
Let $g = ({I_n},v)\in  \SU(n) \ltimes \C^n$ such that $v = [v_{1}, v_{2},\dots, v_n]$.  Assume $n \neq 1$. Then $g$ is strongly reversible in $\SU(n) \ltimes \C^n$.
 \end{corollary} 
\subsection{Proof of \thmref{nsrsu}} 

When $r=0$, i.e. $n=1$, then the result holds trivially. So we assume $n \geq 2$, $r \neq 0$. 

Suppose that $g=(A, v)$ is strongly reversible in $ \SU(n)\ltimes \C^{n}$. Then there exists $ h = (B ,u )$ in $ \SU(n)\ltimes \C^{n}$ such that 
$ hgh^{-1} =g^{-1} $ and $h^2 = I$, where  ${I}$ is the identity in $\SU(n)\ltimes \C^{n}$.  

 Using $hgh^{-1} =g^{-1} $ and $ h^2 ={I}$, we have $B \in \SU(n)$ such that for all  $x \in \C^n$, 
\begin{equation} BAB^{-1}(x) -  BAB^{-1} (u) + B(v) + u = A^{-1}(x) - A^{-1}(v), \end{equation}
\begin{equation} 
B^2(x) + B(u) + u = I_{n}(x). \end{equation} 

From this we have, 
\begin{equation}  \label{eq04}        BAB^{-1} =A^{-1},  ( A^{-1} - I_n)(v-u) = - (B + I_n)(v), \end{equation} 
\begin{equation}  \label{eq05}  \  B ^2 = I_n , (B + I_n )(u) =0.
\end{equation} 
Further, on comparing $n$th row and $n$th column in the matrix equation $BAB^{-1}= A^{-1}$, we get that matrix $B$ has the following form : 
\begin{equation} \label{eq06}
  B = \begin{pmatrix}
B_1 & \\
 & \alpha \\ 
 \end{pmatrix},  \hbox{where}  \   \alpha \in \C \setminus \{0\} \ \hbox{and} \ B _1 \in \GL(n-1,\C ) \end{equation} such that $ B_1 A_1 B_1^{-1} = A_1 ^{-1}$, where 
$ {A_1}  = \textnormal{diag} (e^{i\theta_1},  e^{-i\theta_1}, \dots, e^ {i\theta_{2r}}, e^{-i\theta_{2r}})$.

Note that $v = [0,0, \dots, 0,v_1 ] \in \C^n $ is such that  $v_1 \neq 0$. By using the above block diagonal form of $B$ in the Equation $( A^{-1} - I_n)(v-u) = - (B + I_n)(v)$ and on comparing the last rows, we get $\alpha =-1$. Now, $ B^2 = I_n$ and the Equation \eqref{eq06} implies $B_1A_1B_1^{-1} = A_1^{-1}$ and $B_1^2 = I_{n-1}$. Therefore, from \lemref{le2.3}, we have  $ \det (B_1 ) = (-1)^{n-1}$.  As $n= 4r+1$ is an odd natural number,  
$$\det( B ) = \det(B_1) \det(\alpha) =  (-1)^{n-1}(-1)= (-1)^n= -1.$$

Therefore, if $g$ is of the form as given in the assertion, then $\det B=-1$,  and hence, $g$ can not be strongly reversible in $\SU(n) \ltimes \C^n$. If $g$ is not of the given form, then strong reversibility of $g$ follows from Lemma \ref{lem-strong-SU(n)},  Lemma \ref{srsu-1}, Lemma \ref{lem-srsu-2}, and  \corref{cor-id-SU(n)}. 
 
This proves the theorem.    
\hfill $\square$

\begin{ack}
The authors thank the referee for many comments and suggestions. The authors also thank Ian Short,  Sushil Bhunia and Chandan Maity for their comments on the first draft of this paper. 

\medskip  Tejbir acknowledges full support from the CSIR SRF grant, file No. : 09/947(0113)/2019-EMR-I,  during the course of this work.  Gongopadhyay acknowledges SERB MATRICS grant MTR/2017/000355.

\end{ack}

\end{document}